\documentclass[leqno,12pt]{amsart}

\usepackage{amssymb}
\usepackage{pictexwd,dcpic,amssymb}
\theoremstyle{plain} 
\newtheorem{theorem}{\indent\sc Theorem}[section]
\newtheorem{lemma}[theorem]{\indent\sc Lemma}

\newtheorem{proposition}[theorem]{\indent\sc Proposition}

\theoremstyle{definition} 

\newtheorem{remark}[theorem]{\indent\sc Remark}
\newtheorem{example}[theorem]{\indent\sc Example}

\begin{document}

\title[title for the running head]{Ray class invariants over imaginary quadratic fields} 

\author[F. Author]{Ho Yun Jung} 

\author[S. Author]{Ja Kyung Koo} 

\author[C. Author]{Dong Hwa Shin} 

%
\thanks{ 
This research was supported by Basic Science Research Program
through the NRF of Korea funded by MEST (2010-0001654). The third
named author is partially supported by TJ Park Postdoctoral
Fellowship.}
\address{
Department of Mathematical Sciences, KAIST \endgraf Daejeon 373-1,
Korea} \email{DOSAL@kaist.ac.kr}

\address{
Department of Mathematical Sciences, KAIST \endgraf Daejeon 373-1,
Korea} \email{jkkoo@math.kaist.ac.kr}

\address{
Department of Mathematical Sciences, KAIST \endgraf
Daejeon 373-1,
Korea} \email{shakur01@kaist.ac.kr}


\maketitle

\begin{abstract}
Let $K$ be an imaginary quadratic field of discriminant less than or
equal to $-7$ and $K_{(N)}$ be its ray class field modulo $N$ for an
integer $N$ greater than $1$. We prove that singular values of
certain Siegel functions generate $K_{(N)}$ over $K$ by extending
the idea of our previous work (\cite{J-K-S}). These generators are
not only the simplest ones conjectured by Schertz (\cite{Schertz}),
but also quite useful in the matter of computation of class
polynomials. We indeed give an algorithm to find all conjugates of
such generators by virtue of \cite{Gee} and \cite{Stevenhagen}.
\end{abstract}

\section*{Introduction} 

Let $K$ be an imaginary quadratic field and
$\mathcal{O}_K=\mathbb{Z}[\theta]$ be the ring of integers with
$\theta$ in the complex upper half plane $\mathfrak{H}$. We denote
the Hilbert class field and the ray class field modulo $N$ of $K$
for a positive integer $N$ by $H$ and $K_{(N)}$, respectively. Hasse
(\cite{Hasse} or \cite[Chapter 10 Corollary to Theorem 7]{Lang})
found in 1927 that for any integral ideal $\mathfrak{a}$ of $K$,
$K_{(N)}$ is generated over $H$ by adjoining the value of the Weber
function for the elliptic curve $\mathbb{C}/\mathfrak{a}$ at a
generator of the cyclic $\mathcal{O}_K$-module
$(1/N)\mathfrak{a}/\mathfrak{a}$. It requires good understanding of
the arithmetic of elliptic curves, which is formulated by the theory
of complex multiplication (\cite[Chapter 10]{Lang} or \cite[Chapter
5]{Shimura}). Together with Shimura's reciprocity law which reveals
a remarkable relationship between class field theory and modular
function fields, the theory of Shimura's canonical model allows us
to generate $K_{(N)}$ over $K$ by the specialization of a certain
modular function field. In particular, Cho-Koo (\cite[Corollary
5.2]{C-K}) showed that the singular value of a Hauptmodul with
rational Fourier coefficients on some modular curve generates
$K_{(N)}$ over $K$. For instance, Cho-Koo-Park (\cite[Theorem
13]{C-K-P}) considered the case $N=6$ in terms of Ramanujan's cubic
continued fraction. Also Koo-Shin further provided in
\cite[pp.161--162]{K-S} appropriate Hauptmoduli for this purpose.
\par
It seems to be a difficult problem to construct a ray class
invariant (as a primitive generator of $K_{(N)}$) over $K$ by means
of values of a transcendental function which can be applied to all
$K$ and $N$. In 1964 Ramachandra (\cite[Theorem 10]{Ramachandra}) at
last found universal generators of ray class fields of arbitrary
moduli by applying the Kronecker limit formula. However his
invariants involve overly complicated products of high powers of
singular values of the Klein forms and singular values of the
discriminant $\Delta$-function. On the other hand, Schertz
(\cite[Theorems 3 and 4]{Schertz}) attempted to find simple and
better answers for practical use with similar ideas. The simplest
generators conjectured by Schertz (\cite[p.386]{Schertz}) are
singular values of a Siegel function, and Jung-Koo-Shin
(\cite[Theorem 2.4]{J-K-S}) showed that his conjectural generators
are the right ones at least over $H$ for
$K\neq\mathbb{Q}(\sqrt{-1}),\mathbb{Q}(\sqrt{-3})$.
\par
Since the primitive element theorem guarantees the existence of a
simple generator of $K_{(N)}$ over $K$, one might try to combine
Hasse's two generators to get a ray class invariant. Cho-Koo
(\cite[Corollary 5.5]{C-K}) recently succeeded in obtaining such a
generator by showing that the singular value of a Weber function is
an algebraic integer and then applying the result of Gross-Zagier
(\cite{G-Z} or \cite[Theorem 13.28]{Cox}). Koo-Shin (\cite[Theorems
9.8 and 9.10]{K-S}) further investigated the problem over $K$ in a
completely different point of view by using both singular values of
the elliptic modular function $j$ and Siegel functions.
\par
For any pair $(r_1,r_2)\in\mathbb{Q}^2\setminus\mathbb{Z}^2$ we
define a \textit{Siegel function} $g_{(r_1,r_2)}(\tau)$ on
$\tau\in\mathfrak{H}$ by the following infinite product expansion
\begin{eqnarray}\label{FourierSiegel}
g_{(r_1,r_2)}(\tau)=-q_\tau^{(1/2)\textbf{B}_2(r_1)}e^{\pi
ir_2(r_1-1)}(1-q_z)\prod_{n=1}^{\infty}(1-q_\tau^nq_z)(1-q_\tau^nq_z^{-1}),
\end{eqnarray}
where $\textbf{B}_2(X)=X^2-X+1/6$ is the second Bernoulli
polynomial, $q_\tau=e^{2\pi i\tau}$ and $q_z=e^{2\pi iz}$ with
$z=r_1\tau+r_2$. Then it is a modular unit in the sense of
\cite[p.36]{K-L}. Since its Fourier coefficients are quite small, we
are able to estimate and compare the values of the function in order
to derive our main theorem.
\par
Let $\mathfrak{a}=[\omega_1,\omega_2]$ be a fractional ideal of $K$
not containing $1$, where $\{\omega_1,\omega_2\}$ is an oriented
basis such that $\omega_1/\omega_2\in\mathfrak{H}$. Writing
$1=r_1\omega_1+r_2\omega_2$ for some
$(r_1,r_2)\in\mathbb{Q}^2\setminus\mathbb{Z}^2$ we denote
\begin{equation*}
g(1,[\omega_1,\omega_2])=g_{(r_1,r_2)}(\omega_1/\omega_2).
\end{equation*}
When a product of these values becomes a unit, we call it an
\textit{elliptic unit}. The $12$-th power the above value depends
only on $\mathfrak{a}$ itself (\cite[Chapter 2 Remark to Theorem
1.2]{K-L}). So we may write $g^{12}(1,\mathfrak{a})$ instead of
$g^{12}(1,[\omega_1,\omega_2])$.
\par
For a nontrivial integral ideal $\mathfrak{f}$ of $K$, let
$I_K(\mathfrak{f})$ be the group of fractional ideals of $K$ which
are relatively prime to $\mathfrak{f}$, and $P_{K,1}(\mathfrak{f})$
be the subgroup of $I_K(\mathfrak{f})$ generated by the principal
ideals $\alpha\mathcal{O}_K$ for $\alpha\in\mathcal{O}_K$ which
satisfies $\alpha\equiv1\pmod{\mathfrak{f}}$. Then the ideal class
group $I_K(\mathfrak{f})/P_{K,1}(\mathfrak{f})$ is isomorphic to the
Galois group of the ray class field $K_\mathfrak{f}$ modulo
$\mathfrak{f}$ over $K$ (\cite[pp.116--118]{Silverman}). Now we
consider the value
\begin{equation*}
g^{12N(\mathfrak{f})}(1,\mathfrak{f}),
\end{equation*}
where $N(\mathfrak{f})$ is the smallest positive integer in
$\mathfrak{f}$. It belongs to $K_\mathfrak{f}$ (\cite[Chapter 2
Proposition 1.3 and Chapter 11 Theorem 1.1]{K-L}). Let
$\sigma:I_K(\mathfrak{f})/P_{K,1}(\mathfrak{f})\rightarrow\mathrm{Gal}(K_\mathfrak{f}/K)$
be the Artin map. Then for a ray class $C\in
I_K(\mathfrak{f})/P_{K,1}(\mathfrak{f})$, $\sigma(C)$ satisfies the
rule
\begin{equation}\label{rule}
g^{12N(\mathfrak{f})}(1,\mathfrak{f})^{\sigma(C)}
=g^{12N(\mathfrak{f})}(1,\mathfrak{f}\mathfrak{c}^{-1}),
\end{equation}
where $\mathfrak{c}$ is a representative integral ideal of $C$ by
the theory of complex multiplication (\cite[pp.235--236]{K-L}). In
our case we take $\mathfrak{f}=N\mathcal{O}_K$ for an integer $N$
($\geq2$). In this paper, as Schertz conjectured, we shall show that
the singular value
\begin{equation*}
g^{12N}(1,N\mathcal{O}_K)=g_{(0,1/N)}^{12N}(\theta)
\end{equation*}
alone, or any one of its integral powers generates $K_{(N)}$ over
$K$ ($\neq\mathbb{Q}(\sqrt{-1}),\mathbb{Q}(\sqrt{-3})$) (Theorem
\ref{primitive} and Remark \ref{exception}). While the formula
(\ref{rule}) provides all conjugates of $g_{(0,1/N)}^{12N}(\theta)$,
it is inconvenient for practical use because we can hardly describe
bases of representative ideals in general. Therefore, rather than
working with actions of $\mathrm{Gal}(K_{(N)}/K)$ directly by
(\ref{rule}) we will manipulate actions of $\mathrm{Gal}(H/K)$ and
$\mathrm{Gal}(K_{(N)}/H)$ separately by following Gee-Stevenhagen's
idea (\cite[$\S$3, 9, 10]{Gee} or \cite[$\S$3, 6]{Stevenhagen}).

\section{Fields of modular functions}

This section will be devoted to reviewing briefly modular function
fields and actions of Galois groups in terms of Siegel functions.
For the full description of the modularity of Siegel functions we
refer to \cite{K-S} or \cite{K-L}.
\par
For a positive integer $N$, let $\zeta_N=e^{2\pi i/N}$ and
\begin{eqnarray*}
\Gamma(N)=\bigg\{\begin{pmatrix}a&b\\c&d\end{pmatrix}\in
\mathrm{SL}_2(\mathbb{Z})~;~
\begin{pmatrix}a&b\\c&d\end{pmatrix}\equiv
\begin{pmatrix}1&0\\0&1\end{pmatrix}\pmod{N}\bigg\}
\end{eqnarray*}
be the principal congruence subgroup of level $N$ of
$\mathrm{SL}_2(\mathbb{Z})$. The group $\Gamma(N)$ acts on
$\mathfrak{H}$ by fractional linear transformations, and the orbit
space $Y(N)=\Gamma(N)\backslash\mathfrak{H}$ can be given a
structure of a Riemann surface. Furthermore, $Y(N)$ can be
compactified by adding cusps so that
$X(N)=\Gamma(N)\backslash\mathfrak{H}^*$ with
$\mathfrak{H}^*=\mathfrak{H}\cup\mathbb{P}^1(\mathbb{Q})$ becomes a
compact Riemann surface (or an algebraic curve), which we call the
\textit{modular curve of level $N$} (\cite[Chapter 2]{D-S} or
\cite[$\S$1.5]{Shimura}).
\par
Meromorphic functions on $X(N)$ are called \textit{modular functions
of level $N$}. In particular, we are interested in the field of
modular functions of level $N$ defined over the $N$-th cyclotomic
field $\mathbb{Q}(\zeta_N)$ which is denoted by $\mathcal{F}_N$.
Then it is well-known that the extension
$\mathcal{F}_N/\mathcal{F}_1$ is Galois and
\begin{equation*}
\mathrm{Gal}(\mathcal{F}_N/\mathcal{F}_1)\cong\mathrm{GL}_2(\mathbb{Z}/N\mathbb{Z})/\{\pm1_2\},
\end{equation*}
whose action is given as follows: We can decompose an element
$\alpha\in\mathrm{GL}_2(\mathbb{Z}/N\mathbb{Z})/\{\pm1_2\}$ into
$\alpha=\alpha_1\cdot\alpha_2$ for some
$\alpha_1\in\mathrm{SL}_2(\mathbb{Z}/N\mathbb{Z})/\{\pm1_2\}$ and
$\alpha_2=\left(\begin{smallmatrix}1&0\\0&d\end{smallmatrix}\right)$.
The action of $\alpha_1$ is defined by a fractional linear
transformation. And, $\alpha_2$ acts by the rule
\begin{equation*}
\sum_{n>-\infty} c_nq_\tau^{n/N}\mapsto \sum_{n>-\infty}
c_n^{\sigma_d}q_\tau^{n/N},
\end{equation*}
where $\sum_{n>-\infty} c_nq_\tau^{n/N}$ is the Fourier expansion of
a function in $\mathcal{F}_N$ and $\sigma_d$ is the automorphism of
$\mathbb{Q}(\zeta_N)$ defined by $\zeta_N^{\sigma_d}=\zeta_N^d$
(\cite[Chapter 6 Theorem 3]{Lang}).
\par
It is well-known that the fields $\mathcal{F}_N$ are described by
$j(\tau)$ and the Fricke functions (\cite[Chapter 6 Corollary
1]{Lang} or \cite[Proposition 6.9]{Shimura}). However, we restate
these fields in terms of Siegel functions for later use. First, we
need some transformation formulas and modularity criterion for
Siegel functions. For $x\in\mathbb{R}$ we define $\langle x\rangle$
by the fractional part of $x$ such that $0\leq \langle x\rangle<1$.

\begin{proposition}\label{F_N}
Let $N\geq 2$. For
$(r_1,r_2)\in(1/N)\mathbb{Z}^2\setminus\mathbb{Z}^2$ the function
$g_{(r_1,r_2)}^{12N}(\tau)$ satisfies the relation
\begin{equation*}
g_{(r_1,r_2)}^{12N}(\tau)=g_{(-r_1,-r_2)}^{12N}(\tau)=g_{(\langle
r_1\rangle,\langle r_2\rangle)}^{12N}(\tau).
\end{equation*}
It belongs to $\mathcal{F}_N$ and
$\alpha\in\mathrm{GL}_2(\mathbb{Z}/N\mathbb{Z})/\{\pm1_2\}$ acts on
it by
\begin{equation*}
g_{(r_1,r_2)}^{12N}(\tau)^\alpha= g_{(r_1,r_2)\alpha}^{12N}(\tau).
\end{equation*}
Also we have
\begin{equation*}
\mathcal{F}_N=\mathbb{Q}(\zeta_N)(j(\tau),~g_{(1/N,0)}^{12N}(\tau),~
g_{(0,1/N)}^{12N}(\tau)).
\end{equation*}
\end{proposition}
\begin{proof}
See \cite[Proposition 2.4, Theorems 2.5 and 4.2]{K-S}.
\end{proof}

We set
\begin{equation*}
\mathcal{F}=\bigcup_{N=1}^\infty \mathcal{F}_N.
\end{equation*}
Passing to the projective limit of exact sequences
\begin{equation*}
1\longrightarrow
\{\pm1_2\}\longrightarrow\mathrm{GL}_2(\mathbb{Z}/N\mathbb{Z})\longrightarrow\mathrm{Gal}(\mathcal{F}_N/\mathcal{F}_1)\longrightarrow1
\end{equation*}
for all $N$ ($\geq1$), we obtain an exact sequence
\begin{equation}\label{functionexact}
1\longrightarrow
\{\pm1_2\}\longrightarrow\prod_{p~:~\textrm{primes}}\mathrm{GL}_2(\mathbb{Z}_p)\longrightarrow\mathrm{Gal}(\mathcal{F}/\mathcal{F}_1)\longrightarrow1.
\end{equation}
For every $u=(u_p)_p\in\prod_p\mathrm{GL}_2(\mathbb{Z}_p)$ and a
positive integer $N$, there exists an integral matrix $\alpha$ in
$\mathrm{GL}_2^+(\mathbb{Q})$ with $\det(\alpha)>0$ such that
$\alpha\equiv u_p\pmod{N\mathbb{Z}_p}$ for all $p$ dividing $N$ by
the Chinese remainder theorem. The action of $u$ on $\mathcal{F}_N$
can be described by the action of $\alpha$ (\cite[Proposition
6.21]{Shimura}).

\section{Shimura's reciprocity law}

We shall develop an algorithm for finding all conjugates of the
singular value of a modular function, from which we can determine
all conjugates of $g_{(0,1/N)}^{12N}(\theta)$. To this end we adopt
Gee-Stevenhagen's idea which explains Shimura's reciprocity law
explicitly for practical use.
\par
Let $\mathbb{A}_\mathbb{Q}^\mathrm{f}=\prod_p^\prime\mathbb{Q}_p$
denote the ring of finite adeles. Here, the restricted product is
taken with respect to the subrings
$\mathbb{Z}_p\subset\mathbb{Q}_p$. Every
$x\in\mathrm{GL}_2(\mathbb{A}_\mathbb{Q}^\mathrm{f})$ can be written
as
\begin{equation*}
x=u\cdot\alpha\quad\textrm{with}~u\in\prod_p\mathrm{GL}_2(\mathbb{Z}_p)~\textrm{and}~
\alpha\in\mathrm{GL}_2^+(\mathbb{Q}),
\end{equation*}
since the class number of $\mathbb{Q}$ is one (\cite[Lemma
6.19]{Shimura}). Such a decomposition $x=u\cdot\alpha$ determines a
group action of $\mathrm{GL}_2(\mathbb{A}_\mathbb{Q}^\mathrm{f})$ on
$\mathcal{F}$ by
\begin{equation*}
h^x=h^u\circ\alpha,
\end{equation*}
where $h^u$ is given by the exact sequence (\ref{functionexact})
(\cite[pp.149--150]{Shimura}). Then we have the following
\textit{Shimura's exact sequence}
\begin{equation*}
1\longrightarrow\mathbb{Q}^*\longrightarrow\mathrm{GL}_2(\mathbb{A}_\mathbb{Q}^\mathrm{f})
\longrightarrow\mathrm{Aut}(\mathcal{F})\longrightarrow1
\end{equation*}
(\cite[Theorem 6.23]{Shimura}).
\par
Let $K$ be an imaginary quadratic field of discriminant $d_K$. From
now on we fix
\begin{eqnarray}\label{theta}
\theta=\left\{\begin{array}{ll}\sqrt{d_K}/2&\textrm{for}~d_K\equiv0\pmod{4}\vspace{0.1cm}\\
(-1+\sqrt{d_K})/2&\textrm{for}~d_K\equiv1\pmod{4},\end{array}\right.
\end{eqnarray}
which satisfies $\mathcal{O}_K=\mathbb{Z}[\theta]$. Then we have
\begin{equation*}
\min(\theta,\mathbb{Q})=X^2+B_\theta X+C_\theta
=\left\{\begin{array}{ll} X^2-d_K/4 &
\textrm{if}~d_K\equiv0\pmod{4}\vspace{0.1cm}\\
X^2+X+(1-d_K)/4 & \textrm{if}~d_K\equiv1\pmod{4}.
\end{array}\right.
\end{equation*}
We use the notation $K_p=K\otimes_\mathbb{Q}\mathbb{Q}_p$ for each
prime $p$ and denote the group of finite ideles of $K$ by
$(\mathbb{A}_K^\mathrm{f})^*=\prod_p^\prime K_p^*$, where the
restricted product is taken with respect to the subgroups
$\mathcal{O}_p^*=(\mathcal{O}_K\otimes_\mathbb{Z}\mathbb{Z}_p)^*$ of
$K_p^*$. Let $[\cdot,K]$ denote the Artin map on
$(\mathbb{A}_K^\mathrm{f})^*$. Then the class field theory on $K$ is
summarized in the following exact sequence
\begin{equation*}
1\longrightarrow
K^*\longrightarrow(\mathbb{A}_K^\mathrm{f})^*\stackrel{[\cdot,K]}
{\longrightarrow}\mathrm{Gal}(K^\mathrm{ab}/K)\longrightarrow1,
\end{equation*}
where $K^\mathrm{ab}$ is the maximal abelian extension of $K$
(\cite[$\S$5.2]{Shimura} or \cite[Chapter II Theorem
3.5]{Silverman}). The main theorem of the theory of complex
multiplication states that the value $j(\theta)$ generates $H$ over
$K$, and the sequence
\begin{equation}\label{classfieldexact}
1\longrightarrow
\mathcal{O}_K^*\longrightarrow\prod_p\mathcal{O}_p^*\stackrel
{[\cdot,K]}{\longrightarrow}\mathrm{Gal}(K^\mathrm{ab}/K(j(\theta)))\longrightarrow1
\end{equation}
is exact (\cite[Theorem 5.7]{Shimura}). Furthermore, $K_{(N)}$ is
none other than the field $K(\mathcal{F}_N(\theta))$ which is the
extension field of $K$ obtained by adjoining all singular values
$h(\theta)$ for $h\in\mathcal{F}_N$ which is defined and finite at
$\theta$ (\cite[Chapter 10 Corollary to Theorem 2]{Lang} or
\cite[Proposition 6.33]{Shimura}).
\par
For each prime $p$ we define
\begin{equation*}
(g_\theta)_p:K_p^*\longrightarrow \mathrm{GL}_2(\mathbb{Q}_p)
\end{equation*}
as the injection that sends $x_p\in K_p^*$ to the matrix in
$\mathrm{GL}_2(\mathbb{Q}_p)$ which represents the multiplication by
$x_p$ with respect to the $\mathbb{Q}_p$-basis
$\left(\begin{smallmatrix}\theta\\1\end{smallmatrix}\right)$ for
$K_p$. More precisely, if
$\mathrm{min}(\theta,\mathbb{Q})=X^2+B_\theta X+C_\theta$, then for
$s_p,~t_p\in\mathbb{Q}_p$ we can describe the map as
\begin{equation*}
(g_\theta)_p~:~s_p\theta+t_p\mapsto\begin{pmatrix}t_p-B_\theta
s_p&-C_\theta s_p\\s_p&t_p\end{pmatrix}.
\end{equation*}
On $(\mathbb{A}_K^\mathrm{f})^*$ we have an injection
\begin{equation*}
g_\theta=\prod_p(g_\theta)_p~:~(\mathbb{A}_K^\mathrm{f})^*\longrightarrow{\prod_p}^\prime\mathrm{GL}_2(\mathbb{Q}_p),
\end{equation*}
where the restricted product is taken with respect to the subgroups
$\mathrm{GL}_2(\mathbb{Z}_p)$ of $\mathrm{GL}_2(\mathbb{Q}_p)$.
Combining (\ref{functionexact}) and (\ref{classfieldexact}) we get
the diagram
\begin{eqnarray}\label{exactdiagram}
\begin{array}{ccccccccc}
1& \longrightarrow & \mathcal{O}^* &\longrightarrow
&\prod_p\mathcal{O}_p^*& \stackrel{[\cdot,K]}{\longrightarrow} &
\mathrm{Gal}(K^\mathrm{ab}/K(j(\theta)))
&\longrightarrow &1\\
&&&&\phantom{\bigg\downarrow}\big\downarrow g_\theta&&&&\\
1 & \longrightarrow &\{\pm1_2\} &\longrightarrow
&\prod_p\mathrm{GL}_2(\mathbb{Z}_p)& \longrightarrow &
\mathrm{Gal}(\mathcal{F}/\mathcal{F}_1) &\longrightarrow &1.
\end{array}
\end{eqnarray}
Then \textit{Shimura's reciprocity law} says that for
$h\in\mathcal{F}$ and $x\in\prod_p\mathcal{O}_p^*$
\begin{equation}\label{reciprocity}
h(\theta)^{[x^{-1},K]}=h^{(g_\theta(x))}(\theta)
\end{equation}
(\cite[Theorem 6.31]{Shimura}).
\par
Let $Q=[a,b,c]=aX^2+bXY+cY^2\in\mathbb{Z}[X,Y]$ be a primitive
positive definite quadratic form of discriminant $d_K$. Under an
appropriate equivalence relation these forms determine the group
$\mathrm{C}(d_K)$, called the \textit{form class group of
discriminant $d_K$}. In particular, the unit element is the class
containing
\begin{eqnarray}\label{unitform}
\left\{\begin{array}{ll} \textrm{[}1,0,-d_K/4\textrm{]} &
\textrm{for}~d_K\equiv0\pmod{4}\vspace{0.1cm}\\
\textrm{[}1,1,(1-d_K)/4\textrm{]} & \textrm{for}~
d_K\equiv1\pmod{4},\end{array}\right.
\end{eqnarray}
and the inverse of the class containing $[a,b,c]$ is the class
containing $[a,-b,c]$ (\cite[Theorem 3.9]{Cox}). We identify
$\mathrm{C}(d_K)$ with the set of all \textit{reduced quadratic
forms}, which are characterized by the conditions
\begin{equation}\label{reduced}
(-a<b\leq a<c\quad\textrm{or}\quad 0\leq b\leq
a=c)\quad\textrm{and}\quad b^2-4ac=d_K
\end{equation}
(\cite[Theorem 2.9]{Cox}). Note that the above two conditions for
reduced quadratic forms imply
\begin{equation}\label{bound a}
a\leq\sqrt{-d_K/3}
\end{equation} (\cite[p.29]{Cox}).
It is well-known that $\mathrm{C}(d_K)$ is isomorphic to
$\mathrm{Gal}(H/K)$ (\cite[Theorem 7.7]{Cox}).  Gee and Stevenhagen
found an idele $x_Q\in(\mathbb{A}_K^\mathrm{f})^*$ such that
\begin{equation*}
[x_Q,K]|_H=[a,b,c].
\end{equation*}

\begin{proposition}\label{idele}
Let $Q=[a,b,c]$ be a primitive positive definite quadratic form of
discriminant $d_K$. We put
\begin{equation*}
\theta_Q=(-b+\sqrt{d_K})/2a.
\end{equation*}
Furthermore, for each prime $p$ we define $x_p$ as
\begin{eqnarray*}
x_p=\left\{\begin{array}{ll}a&\textrm{if}~p\nmid
a\vspace{0.1cm}\\
a\theta_Q&\textrm{if}~p\mid a~\textrm{and}~p\nmid c\vspace{0.1cm}\\
a(\theta_Q-1)&\textrm{if}~p\mid a~\textrm{and}~p\mid c.
\end{array}\right.
\end{eqnarray*}
Then for $x_Q=(x_p)_p\in(\mathbb{A}_K^\mathrm{f})^*$ the Galois
action of the Artin symbol $[x_Q,K]$ satisfies the relation
\begin{equation*}
j(\theta)^{[a,b,c]}=j(\theta)^{[x_Q,K]}.
\end{equation*}
\end{proposition}
\begin{proof}
See \cite[Lemma 19]{Gee}  or \cite[$\S$6]{Stevenhagen}.
\end{proof}

The next proposition gives the action of $[x_Q^{-1},K]$ on
$K^\mathrm{ab}$ by using Shimura's reciprocity law
(\ref{reciprocity}).

\begin{proposition}\label{Hilbertclass}
Let $Q=[a,b,c]$ be a primitive positive definite quadratic form of
discriminant $d_K$ and $\theta_Q$  be as in Proposition
\textup{\ref{idele}}. Define
$u_Q=(u_p)_p\in\prod_p\mathrm{GL}_2(\mathbb{Z}_p)$ as
\begin{itemize}
\item[] Case 1 : $d_K\equiv0\pmod{4}$
\begin{eqnarray}\label{u1}
u_p=\left\{\begin{array}{ll}
\begin{pmatrix}a&b/2\\0&1\end{pmatrix}&\textrm{if}~p\nmid a\vspace{0.1cm}\\
\begin{pmatrix}-b/2&-c\\1&0\end{pmatrix}&\textrm{if}~p\mid a~\textrm{and}~p\nmid c\vspace{0.1cm}\\
\begin{pmatrix}-a-b/2&-c-b/2\\1&-1\end{pmatrix}&\textrm{if}~p\mid a~\textrm{and}~p\mid c
\end{array}\right.
\end{eqnarray}
\item[] Case 2 : $d_K\equiv1\pmod{4}$
\begin{eqnarray}\label{u2}
u_p=\left\{\begin{array}{ll}
\begin{pmatrix}a&(b-1)/2\\0&1\end{pmatrix}&\textrm{if}~p\nmid a\vspace{0.1cm}\\
\begin{pmatrix}-(b+1)/2&-c\\1&0\end{pmatrix}&\textrm{if}~p\mid a~\textrm{and}~p\nmid c\vspace{0.1cm}\\
\begin{pmatrix}-a-(b+1)/2&-c+(1-b)/2\\1&-1\end{pmatrix}&\textrm{if}~p\mid a~\textrm{and}~p\mid
c.
\end{array}\right.
\end{eqnarray}
\end{itemize}
Then for $h\in\mathcal{F}$ which is defined and finite at $\theta$
we have
\begin{equation*}
h(\theta)^{[x_Q^{-1},K]}=h^{u_Q}(\theta_Q).
\end{equation*}
\end{proposition}
\begin{proof}
See \cite[Lemma 20]{Gee} or \cite[$\S$6]{Stevenhagen}.
\end{proof}

For each positive integer $N$ we define the matrix group
\begin{equation*}
W_{N,\theta}=\bigg\{\begin{pmatrix}t-B_\theta s & -C_\theta
s\\s&t\end{pmatrix}\in\mathrm{GL}_2(\mathbb{Z}/N\mathbb{Z})~;~t,~s\in\mathbb{Z}/N\mathbb{Z}\bigg\}.
\end{equation*}
Analyzing the diagram (\ref{exactdiagram}) and using Shimura's
reciprocity law (\ref{reciprocity}), Gee and Stevenhagen could
express $\mathrm{Gal}(K_{(N)}/H)$ quite explicitly.

\begin{proposition}\label{rayclass}
Assume that $K\neq\mathbb{Q}(\sqrt{-1}),\mathbb{Q}(\sqrt{-3})$ and
$N\geq1$. Then we have a surjection
\begin{eqnarray*}
W_{N,\theta}&\longrightarrow&\mathrm{Gal}(K_{(N)}/H)\\
\label{action}\alpha&\longmapsto&(h(\theta)\mapsto
h^\alpha(\theta))~\textrm{for $h\in \mathcal{F}_N$ which is defined
and finite at}~\theta,
\end{eqnarray*}
whose kernel is $\{\pm 1_2\}$ by \textup{(\ref{exactdiagram})} and
\textup{(\ref{reciprocity})}.
\end{proposition}
\begin{proof}
See \cite[pp.50--51]{Gee} or \cite[$\S$3]{Stevenhagen}.
\end{proof}

Finally we obtain an assertion which we shall use to solve our main
problem. In the next theorem, we follow the notations in
Propositions \ref{Hilbertclass} and \ref{rayclass}.

\begin{theorem}\label{conjugate}
Assume that $K\neq\mathbb{Q}(\sqrt{-1}),\mathbb{Q}(\sqrt{-3})$ and
$N\geq1$. Then there is a one-to-one correspondence
\begin{eqnarray*}
W_{N,\theta}/\{\pm1_2\}\times\mathrm{C}(d_K)&\longrightarrow&\mathrm{Gal}(K_{(N)}/K)\\
(\alpha,Q)&\mapsto&(h(\theta)\mapsto h^{\alpha\cdot
u_Q}(\theta_Q))~\textrm{for $h\in\mathcal{F}_N$ which is defined and
finite at $\theta$.}
\end{eqnarray*}
Here, we follow notations in Propositions
\textup{\ref{Hilbertclass}} and \textup{\ref{rayclass}}.
\end{theorem}
\begin{proof}
Observe the following diagram:
\begin{equation*}
\begindc{0}[3]

\obj(0,0)[A]{$K$}

\obj(0,15)[B]{$H$}

\obj(0,30)[C]{$K_{(N)}$}

\obj(0, 35)[D]{\underline{Fields}}

\obj(30, 35)[E]{\underline{Galois groups}}
\mor{A}{B}{$\quad\Bigg)\quad\mathrm{Gal}(H/K)=\{[x_Q,~K]|_H~;~Q\in
\mathrm{C}(d_K)\}\quad\textrm{by Proposition
\ref{idele}}.$}[\atright,\solidline]

\mor{B}{C}{$\quad\Bigg)\quad\mathrm{Gal}(K_{(N)}/H)\cong
W_{N,\theta}/\{\pm1_2\}\quad\textrm{by Proposition
\ref{rayclass}}$}[\atright,\solidline]
\enddc
\end{equation*}
Now the conclusion follows from Proposition \ref{Hilbertclass}.
\end{proof}

\begin{remark}\label{identitycorrespond}
In particular, the unit element of
$W_{N,\theta}/\{\pm1_2\}\times\mathrm{C}(d_K)$ corresponds to the
unit element of $\mathrm{Gal}(K_{(N)}/K)$ by the definitions of
$u_Q$ and $\theta_Q$. Note that the correspondence is not a group
homomorphism.
\end{remark}

\section{Ray class invariants}

In this last section we shall prove that the singular value
$g_{(0,1/N)}^{12N}(\theta)$ generates $K_{(N)}$ by showing that the
only automorphism of $K_{(N)}$ over $K$ which fixes it is the unit
element. Then Galois theory guarantees our theorem.
\par
Throughout this section we let $K$
($\neq\mathbb{Q}(\sqrt{-1}),\mathbb{Q}(\sqrt{-3})$) be an imaginary
quadratic field of discriminant $d_K$ such that $d_K\leq-7$. We put
$D=\sqrt{-d_K/3}$ and define $\theta$, $\theta_Q$, $u_Q$ for each
primitive positive definite quadratic form $Q=[a,b,c]$ as in
(\ref{theta}) and Proposition \ref{Hilbertclass}. If we set
\begin{equation*}
B=|q_\theta|=|e^{2\pi i\theta}|=e^{-\pi\sqrt{-d_K}},
\end{equation*}
then we have
\begin{equation}\label{B}
B\leq e^{-\sqrt{7}\pi}\quad\textrm{and}\quad
B^{1/D}=e^{-\sqrt{3}\pi}.
\end{equation}
\par
In what follows we shall often use the following basic inequality
\begin{equation}\label{basicinequality}
1+X<e^X\quad\textrm{for}~X>0.
\end{equation}

\begin{lemma}\label{ineq}
We have the following inequalities:
\begin{itemize}
\item[(i)] If $N\geq21$, then
$|(1-\zeta_N)/(1-B^{1/DN})|<1.306$.
\item[(ii)] If $N\geq2$, then
$|(1-\zeta_N)/(1-\zeta_N^s)|\leq1$ for all $s\in\mathbb{Z}\setminus
N\mathbb{Z}$.
\item[(iii)] If $N\geq4$,
then $|(1-\zeta_N)/(1-\zeta_N^s)|\leq 1/\sqrt{2}$ for $2\leq s\leq
N/2$.
\item[(iv)] If $N\geq2$, then
$B^{(1/2)(\mathbf{B}_2(0)-\mathbf{B}_2(1/N))}
|(1-\zeta_N)/(1-B^{1/N})|<0.76$.
\item[(v)] $1/(1-B^{X/D})<1+B^{X/1.03D}$
for all $X\geq1/2$.
\item[(vi)] $1/(1-B^X)<1+B^{X/1.03}$ for all
$X\geq1/2$.
\end{itemize}
\end{lemma}
\begin{proof}
(i) It is routine to check that $|(1-\zeta_N)/(1-B^{1/DN})|=
2\sin(\pi/N)/(1-e^{-\sqrt{3}\pi/N})$ is a decreasing function for
$N\geq21$. Hence its value is maximal when
$N=21$, which is less than $1.306$.\\
(ii) $|(1-\zeta_N)/(1-\zeta_N^s)|=|\sin(\pi/N)/\sin(\pi s/N)|\leq1$
for all $s\in\mathbb{Z}\setminus N\mathbb{Z}$.\\
(iii) If $N\geq4$ and $2\leq s\leq N/2$, then
$|\sin(s\pi/N)|\geq\sin(2\pi/N)$. Thus
\begin{eqnarray*}
\bigg|\frac{1-\zeta_N}{1-\zeta_N^s}\bigg|=
\bigg|\frac{\sin(\pi/N)}{\sin(s\pi/N)}\bigg|\leq
\frac{\sin(\pi/N)}{\sin(2\pi/N)}=\frac{1}{2\cos(\pi/N)}\leq
\frac{1}{2\cos(\pi/4)}=\frac{1}{\sqrt{2}}.
\end{eqnarray*}
(iv) Observe that
\begin{eqnarray*}
B^{(1/2)(\mathbf{B}_2(0)-\mathbf{B}_2(1/N))}\bigg|\frac{1-\zeta_N}{1-B^{1/N}}\bigg|
\leq e^{(-\sqrt{7}\pi/2)(1/N-1/N^2)}\frac{2\sin(\pi/N)}
{1-e^{-\sqrt{7}\pi/N}}\quad\textrm{by (\ref{B})}.
\end{eqnarray*}
It is also routine to check the last term on $N$ ($\geq2$) is less
than
$0.76$.\\
(v) By (\ref{B}) the inequality is equivalent to $e^{-\sqrt{3}\pi
X}+e^{-3\sqrt{3}\pi X/103}<1$,
which obviously holds for $X\geq1/2$.\\
(vi) The given inequality is equivalent to $B^X+B^{3X/103}<1$. By
(\ref{B}) it suffices to show $e^{-\sqrt{7}\pi X}+e^{-3\sqrt{7}\pi
X/103}<1$, which is true for all $X\geq1/2$.
\end{proof}

\begin{lemma}\label{newlemma1}
Let $N\geq21$ and $Q=[a,b,c]$ be a reduced quadratic form of
discriminant $d_K$. If $a\geq2$, then the inequality
\begin{equation*}
|g_{(0,1/N)}(\theta)|< |g_{(r/N,s/N)}(\theta_Q)|
\end{equation*}
holds for $(r,s)\in\mathbb{Z}^2\setminus N\mathbb{Z}^2$.
\end{lemma}
\begin{proof}
We may assume $0\leq r\leq N/2$ by Proposition \ref{F_N}. And, note
that $2\leq a\leq D$ by (\ref{bound a}). From (\ref{FourierSiegel})
we obtain that
\begin{eqnarray*}
&&\bigg|\frac{g_{(0,1/N)}(\theta)}{g_{(r/N,s/N)}(\theta_Q)}\bigg|
=\bigg|\frac{g_{(0,1/N)}(\theta)}{g_{(r/N,s/N)}((-b+\sqrt{d_K})/2a)}
\bigg|\\
&\leq&B^{(1/2)(\mathbf{B}_2(0)-(1/a)\mathbf{B}_2(r/N))}
\bigg|\frac{1-\zeta_N}{1-e^{2\pi
i((r/N)(-b+\sqrt{d_K})/2a+s/N)}}\bigg|\\
&&\times\prod_{n=1}^\infty
\frac{(1+B^n)^2}{(1-B^{(1/a)(n+r/N)})(1-B^{(1/a)(n-r/N)})}.
\end{eqnarray*}
If $r\neq0$, then by the fact $2\leq a\leq D$ and Lemma
\ref{ineq}(i),
\begin{eqnarray*}
\bigg|\frac{1-\zeta_N}{1-e^{2\pi
i((r/N)(-b+\sqrt{d_K})/2a+s/N)}}\bigg| \leq
\bigg|\frac{1-\zeta_N}{1-B^{r/Na}}\bigg| \leq
\bigg|\frac{1-\zeta_N}{1-B^{1/ND}}\bigg|<1.306.
\end{eqnarray*}
If $r=0$, then by Lemma \ref{ineq}(ii),
\begin{eqnarray*}
\bigg|\frac{1-\zeta_N}{1-e^{2\pi
i((r/N)(-b+\sqrt{d_K})/2a+s/N)}}\bigg|=
\bigg|\frac{1-\zeta_N}{1-\zeta_N^s}\bigg| \leq 1.
\end{eqnarray*}
Therefore,
\begin{eqnarray*}
&&\bigg|\frac{g_{(0,1/N)}(\theta)}{g_{(r/N,s/N)}(\theta_Q)}
\bigg|\\
&<& B^{(1/2)(\mathbf{B}_2(0)-(1/2)\mathbf{B}_2(0))}\cdot
1.306\cdot\prod_{n=1}^\infty\frac{(1+B^{n})^2}
{(1-B^{n/D})(1-B^{(1/D)(n-1/2)})}\\
&&\textrm{since $2\leq a\leq D$, $0\leq r\leq N/2$}\\
&<& 1.306B^{1/24}\prod_{n=1}^\infty (1+B^{n})^2(1+B^{n/1.03D})
(1+B^{(1/1.03D)(n-1/2)})\quad\textrm{by Lemma \ref{ineq}(v)}\\
&<& 1.306B^{1/24}\prod_{n=1}^\infty
e^{2B^{n}+B^{n/1.03D}+B^{(1/1.03D)(n-1/2)}}
\quad\textrm{by (\ref{basicinequality})}\\
&=&1.306B^{1/24}e^{2B/(1-B)+
(B^{1/1.03D}+B^{1/2.06D})/(1-B^{1/1.03D})}\\
&\leq& 1.306e^{-\sqrt{7}\pi/24}
e^{2e^{-\sqrt{7}\pi}/(1-e^{-\sqrt{7}\pi})+
(e^{-\sqrt{3}\pi/1.03}+e^{-\sqrt{3}\pi/2.06})/
(1-e^{-\sqrt{3}\pi/1.03})}<1\quad\textrm{by (\ref{B})}.
\end{eqnarray*}
This proves the lemma.
\end{proof}

\begin{lemma}\label{newlemma2}
 Let $N\geq2$ and $Q=[1,b,c]$ be a reduced quadratic form of discriminant $d_K$. Then
we have the inequality
\begin{equation*}
|g_{(0,1/N)}(\theta)|< |g_{(r/N,s/N)}(\theta_Q)|
\end{equation*}
for $r,~s\in\mathbb{Z}$ with $r\not\equiv0\pmod{N}$.
\end{lemma}
\begin{proof}
We may assume $1\leq r\leq N/2$ by Proposition \ref{F_N}. Then
\begin{eqnarray*}
&&\bigg|\frac{g_{(0,1/N)}(\theta)}{g_{(r/N,s/N)}(\theta_Q)}
\bigg| \\
&\leq& B^{(1/2)(\mathbf{B}_2(0)-\mathbf{B}_2(r/N))}
\bigg|\frac{1-\zeta_N}{1-B^{r/N}}\bigg|
\prod_{n=1}^\infty\frac{(1+B^n)^2}{{(1-B^{n+r/N})}
{(1-B^{n-r/N})}}\quad\textrm{by (\ref{FourierSiegel})}\\
&<& B^{(1/2)(\mathbf{B}_2(0)-\mathbf{B}_2(1/N))}
\bigg|\frac{1-\zeta_N}{1-B^{1/N}}\bigg|\prod_{n=1}^\infty\frac{(1+B^n)^2}{{(1-B^{n})}
{(1-B^{n-1/2})}}\\
&<& 0.76\prod_{n=1}^\infty (1+B^n)^2(1+B^{n/1.03})
(1+B^{(1/1.03)(n-1/2)})\quad\textrm{by Lemma \ref{ineq}(iv) and (vi)}\\
&<& 0.76\prod_{n=1}^\infty
e^{2B^{n}+B^{n/1.03}+B^{(1/1.03)(n-1/2)}}\quad\textrm{by}~(\ref{basicinequality})\\
&=&0.76 e^{2B/(1-B)+
(B^{1/1.03}+B^{1/2.06})/(1-B^{1/1.03})}\\
&\leq& 0.76e^{2e^{-\sqrt{7}\pi}/(1-e^{-\sqrt{7}\pi})+
(e^{-\sqrt{7}\pi/1.03}+e^{-\sqrt{7}\pi/2.06})/
(1-e^{-\sqrt{7}\pi/1.03})}<1\quad\textrm{by (\ref{B})}.
\end{eqnarray*}
\end{proof}

\begin{lemma}\label{newlemma3}
Let $N\geq2$ and $Q=[1,b,c]$ be a reduced quadratic form of
discriminant $d_K$. Then
\begin{equation*}
|g_{(0,1/N)}(\theta)|<|g_{(0,s/N)}(\theta_Q)|
\end{equation*}
for $s\in\mathbb{Z}$ with $s\not\equiv0,~\pm1\pmod{N}$.
\end{lemma}
\begin{proof}
If $N=2$ or $3$, there is nothing to prove. Thus, let $N\geq4$. Here
we may assume that $2\leq s\leq N/2$ by Proposition \ref{F_N}.
Observe that
\begin{eqnarray*}
&&\bigg|\frac{g_{(0,1/N)}(\theta)}
{g_{(0,s/N)}(\theta_Q)}\bigg|\\
&\leq& \bigg|\frac{1-\zeta_N}{1-\zeta_N^s}\bigg|
\prod_{n=1}^\infty\frac{(1+B^n)^2}{(1-B^n)^2}\quad\textrm{by (\ref{FourierSiegel})}\\
&<& (1/\sqrt{2})\prod_{n=1}^\infty
(1+B^{n})^2(1+B^{n/1.03})^2\quad\textrm{by Lemma \ref{ineq}(iii) and (vi)}\\
&\stackrel{}{<}& \tfrac{1}{\sqrt{2}}\prod_{n=1}^\infty
e^{2B^{n}+2B^{n/1.03}}\quad\textrm{by}~(\ref{basicinequality})\\
&\stackrel{}{=}&
(1/\sqrt{2})e^{2B/(1-B)+2B^{1/1.03}/(1-B^{1/1.03})}\\
&\leq&(1/\sqrt{2})e^{2e^{-\sqrt{7}\pi}/(1-e^{-\sqrt{7}\pi})
+2e^{-\sqrt{7}\pi/1.03}/(1-e^{-\sqrt{7}\pi/1.03})}<1 \quad\textrm{by
(\ref{B})},
\end{eqnarray*}
which proves the lemma.
\end{proof}

Now we are ready to prove our main theorem.

\begin{theorem}\label{primitive}
Let $N\geq21$. Let $K$
\textup($\neq\mathbb{Q}(\sqrt{-1}),\mathbb{Q}(\sqrt{-3})$\textup) be
an imaginary quadratic field and $\theta$ be as in
\textup{(\ref{theta})}. Then for any positive integer $n$ the value
\begin{equation*}
g^{12Nn}(1,N\mathcal{O}_K)=g_{(0,1/N)}^{12Nn}(\theta)
\end{equation*}
generates $K_{(N)}$ over $K$. It is a real algebraic integer and its
minimal polynomial has integer coefficients. In particular, if $N$
has at least two prime factors, then it is an elliptic unit.
\end{theorem}
\begin{proof}
For simplicity we put $g(\tau)=g_{(0,1/N)}^{12Nn}(\tau)$. Since $g$
belongs to $\mathcal{F}_N$ by Proposition \ref{F_N}, the value
$g(\theta)$ lies in $K_{(N)}$ by the main theorem of the theory of
complex multiplication. Hence, if we show that the only element of
$\mathrm{Gal}(K_{(N)}/K)$ fixing the value $g(\theta)$ is the unit
element, then we can conclude that it generates $K_{(N)}$ over $K$
by Galois theory.
\par By Theorem \ref{conjugate} any conjugate of $g(\theta)$ is of
the form $g^{\alpha\cdot u_Q}(\theta_Q)$ for some
$\alpha=\left(\begin{smallmatrix}t-B_\theta s&-C_\theta
s\\s&t\end{smallmatrix}\right)\in W_{N,\theta}$ and a reduced
quadratic form $Q=[a,b,c]$ of discriminant $d_K$. Assume that
$g(\theta)=g^{\alpha\cdot u_Q}(\theta_Q)$. Then Lemma
\ref{newlemma1} implies that $a=1$, which yields
\begin{eqnarray*}
u_Q=\left\{\begin{array}{ll}\begin{pmatrix}1&b/2\\0&1\end{pmatrix}
&\textrm{for}~d_K\equiv0\pmod{4}\vspace{0.1cm}\\
\begin{pmatrix}1&(b-1)/2\\0&1\end{pmatrix}
&\textrm{for}~d_K\equiv1\pmod{4}\end{array}\right.
\end{eqnarray*}
as an element of $\mathrm{GL}_2(\mathbb{Z}/N\mathbb{Z})$ by
(\ref{u1}) and (\ref{u2}). It follows from Proposition \ref{F_N}
that
\begin{eqnarray*}
g(\theta)=g^{\alpha\cdot u_Q}(\theta_Q)=g_{(0,1/N)\alpha
u_Q}^{12Nn}(\theta_Q)= \left\{\begin{array}{ll} g_{(s/N,(s/N)(b/2)+t/N)}^{12Nn}(\theta_Q)&\textrm{for}~d_K\equiv0\pmod{4}\vspace{0.1cm}\\
g_{(s/N,(s/N)(b-1)/2+t/N)}^{12Nn}(\theta_Q)&\textrm{for}~
d_K\equiv1\pmod{4},
\end{array}\right.
\end{eqnarray*}
from which we get $s\equiv0\pmod{N}$ by Lemma \ref{newlemma2}. Now
Lemma \ref{newlemma3} implies that $t\equiv\pm1\pmod{N}$, which
shows that $\alpha$ is the unit element of
$W_{N,~\theta}/\{\pm1_2\}$. Finally (\ref{reduced})implies that
\begin{eqnarray*}
Q=[a,b,c]=\left\{\begin{array}{ll} \textrm{[}1,0,-d_K/4\textrm{]} &
\textrm{for}~d_K\equiv0\pmod{4}\vspace{0.1cm}\\
\textrm{[}1,1,(1-d_K)/4\textrm{]} & \textrm{for}~
d_K\equiv1\pmod{4},\end{array}\right.
\end{eqnarray*}
which represents the unit element of $\mathrm{C}(d_K)$ by
(\ref{unitform}). This implies that $(\alpha,Q)\in
W_{N,\theta}/\{\pm1_2\}\times\mathrm{C}(d_K)$ represents the unit
element of $\mathrm{Gal}(K_{(N)}/K)$ by Remark
\ref{identitycorrespond}. Therefore $g(\theta)$ actually generates
$K_{(N)}$ over $K$.
\par
From (\ref{FourierSiegel}) we have
\begin{eqnarray*}
g(\theta)&=&\bigg(q_\theta^{1/12}(1-\zeta_N)\prod_{n=1}^\infty(1-q_\theta^n\zeta_N)(1-q_\theta^n\zeta_N^{-1})\bigg)^{12Nn}\\
&=&q_\theta^{Nn}(2\sin(\pi/N))^{12Nn}
\prod_{n=1}^\infty(1-(\zeta_N+\zeta_N^{-1})q_\theta^n+q_\theta^{2n})^{12Nn},
\end{eqnarray*}
and this shows that $g(\theta)$ is a real number. Furthermore, we
see from \cite[$\S$3]{K-S} that the function $g(\tau)$ is integral
over $\mathbb{Z}[j(\tau)]$. Since $j(\theta)$ is a real algebraic
integer (\cite[Chapter 5 Theorem 4]{Lang}), so is the value
$g(\theta)$, and its minimal polynomial over $K$ has integer
coefficients. In particular, if $N$ has at least two prime factors,
the function $1/g(\tau)$ is also integral over $\mathbb{Z}[j(\tau)]$
(\cite[Chapter 2 Theorem 2.2]{K-L}); hence $g(\theta)$ becomes a
unit.
\end{proof}

\begin{remark}\label{exception}
\begin{itemize}
\item[(i)] If we assume that
\begin{equation}\label{newcondition}
( N=2,~d_K\leq-43)\quad\textrm{or}\quad
(N=3,~d_K\leq-39)\quad\textrm{or}\quad (N\geq4,~d_K\leq-31),
\end{equation}
then the upper bounds of the inequalities appeared in Lemma
\ref{ineq} should be slightly changed. But it is routine to check
that Lemmas \ref{newlemma1}--\ref{newlemma3} are also true.
Therefore we can establish Theorem \ref{primitive} again under the
condition (\ref{newcondition}), however, we shall not repeat the
similar proofs.
\item[(ii)]
Theorem \ref{primitive} is still valid for all $N\geq2$ and
$K\neq\mathbb{Q}(\sqrt{-1}),\mathbb{Q}(\sqrt{-3})$. Indeed, for the
remaining finite cases $(N=2,~-40\leq d_K\leq-7)$, $(N=3,~-35\leq
d_K\leq-7)$, $(4\leq N\leq20,~-24\leq d_K\leq-7)$ one can readily
verify Lemmas \ref{newlemma1}--\ref{newlemma3} by Theorem
\ref{conjugate} and some numerical estimation, not by using Lemma
\ref{ineq}.
\end{itemize}
\end{remark}

\begin{remark}\label{exponent}
\begin{itemize}
\item[(i)] For $N\geq2$ and $(r_1,r_2)\in(1/N)\mathbb{Z}^2\setminus
\mathbb{Z}^2$, the function $g_{(r_1,r_2)}^{12N/\gcd(6,N)}(\tau)$
belongs to $\mathcal{F}_N$ and satisfies the same transformation
formulas as in Proposition \ref{F_N} by \cite[Theorem 2.5 and
Proposition 2.4]{K-S}. Hence we are able to replace the value
$g_{(0,1/N)}^{12Nn}(\theta)$ in Theorem \ref{primitive} by
$g_{(0,1/N)}^{12Nn/\gcd(6,N)}(\theta)$ with smaller exponent, which
enables us to have class polynomials with relatively small
coefficients.
\item[(ii)] Nevertheless, the exponent of
$g_{(0,1/N)}^{12N/\gcd(6,N)}(\theta)$ could be quite high for
numerical computations. So one usually takes suitable products of
Siegel functions with lower exponents (\cite{B-S}).
\item[(iii)] In order to prove that the singular value
$g_{(0,1/N)}^{12N/\gcd(6,N)}(\theta)$ is a unit, it suffices to
check whether $N$ has more than one prime ideal factor in $K$
(\cite[$\S$6]{Ramachandra}).
\end{itemize}
\end{remark}

Now we close this section by presenting an example which illustrates
Theorem \ref{primitive}, Remarks \ref{exception} and \ref{exponent}.

\begin{example}
Let $K=\mathbb{Q}(\sqrt{-10})$ and $N=6$ ($=2\cdot3$). Then
$d_K=-40$ and $\theta=\sqrt{-10}$. The reduced quadratic forms of
discriminant $d_K$ are
\begin{equation*}
Q_1=[1,~0,~10]\quad\textrm{and}\quad Q_2=[2,~0,~5].
\end{equation*}
So we have
\begin{eqnarray*}
\theta_{Q_1}=\sqrt{-10},~u_{Q_1}=\begin{pmatrix}1&0\\0&1\end{pmatrix}\quad\textrm{and}\quad
\theta_{Q_2}=\sqrt{-10}/2,~u_{Q_2}=\begin{pmatrix}2&-3\\3&4\end{pmatrix}.
\end{eqnarray*}
Furthermore, one can compute the group $W_{6,\theta}/\{\pm1_2\}$
easily and the result is as follows:
\begin{eqnarray*}
W_{6,\theta}/\{\pm1_2\}=\bigg\{
\begin{pmatrix}1&0\\0&1\end{pmatrix},
\begin{pmatrix}1&2\\1&1\end{pmatrix},
\begin{pmatrix}1&4\\2&1\end{pmatrix},
\begin{pmatrix}1&0\\3&1\end{pmatrix},
\begin{pmatrix}1&2\\4&1\end{pmatrix},
\begin{pmatrix}1&4\\5&1\end{pmatrix},
\begin{pmatrix}3&2\\1&3\end{pmatrix},
\begin{pmatrix}3&4\\2&3\end{pmatrix}
\bigg\}.
\end{eqnarray*}
\,Thus the class polynomial is
\begin{eqnarray*}
&&\mathrm{min}(g_{(0,1/6)}^{12}(\theta),K)=
\prod_{r=1}^2\prod_{\alpha\in
W_{6,\theta}/\{\pm1_2\}}(X-g_{(0,1/6)\alpha
u_{Q_r}}^{12}(\theta_{Q_r}))\\
&=&(X-g_{(0,1/6)}^{12}(\sqrt{-10}))
(X-g_{(1/6,1/6)}^{12}(\sqrt{-10}))
(X-g_{(2/6,1/6)}^{12}(\sqrt{-10}))\\
&&(X-g_{(3/6,1/6)}^{12}(\sqrt{-10}))
(X-g_{(4/6,1/6)}^{12}(\sqrt{-10}))
(X-g_{(5/6,1/6)}^{12}(\sqrt{-10}))\\
&&(X-g_{(1/6,3/6)}^{12}(\sqrt{-10}))
(X-g_{(2/6,3/6)}^{12}(\sqrt{-10}))
(X-g_{(3/6,4/6)}^{12}(\sqrt{-10}/2))\\
&&(X-g_{(5/6,1/6)}^{12}(\sqrt{-10}/2))
(X-g_{(1/6,4/6)}^{12}(\sqrt{-10}/2))
(X-g_{(3/6,1/6)}^{12}(\sqrt{-10}/2))\\
&&(X-g_{(5/6,4/6)}^{12}(\sqrt{-10}/2))
(X-g_{(1/6,1/6)}^{12}(\sqrt{-10}/2))
(X-g_{(5/6,3/6)}^{12}(\sqrt{-10}/2))\\
&&(X-g_{(1/6,0)}^{12}(\sqrt{-10}/2))\\
&=&X^{16}+20560X^{15}-1252488X^{14}-829016560X^{13}-8751987701092X^{12}\\
&&+217535583987600X^{11}+181262520621110344X^{10}+43806873084101200X^9\\
&&-278616280004972730X^8+139245187265282800X^7-8883048242697656X^6\\
&&+352945014869040X^5+23618989732508X^4-1848032773840X^3+49965941112X^2\\
&&-425670800X+1,
\end{eqnarray*}
which shows that $g_{(0,1/6)}^{12}(\theta)$ is also a unit.
\end{example}



\begin{thebibliography}{99}


\bibitem {B-S} \textsc{S. Bettner and R. Schertz}, Lower
powers of elliptic units,  J. Th\'{e}or. Nombres Bordeaux 13 (2001),
339--351.

\bibitem{C-K} \textsc{B. Cho and J. K. Koo}, Constructions of ray class fields over imaginary quadratic fields and applications,
Q. J. Math. 61 (2010), 199--216.

\bibitem{C-K-P} \textsc{B. Cho, J. K. Koo and Y. K. Park}, On
Ramanujan's cubic continued fraction as modular function, Tohoku
Math. J. 62 (2010), 579--603.

\bibitem{Cox} \textsc{D. A. Cox}, Primes of the form $x^2+ny^2$,
Fermat, class field, and complex multiplication, John Wiley \& Sons,
Inc., New York, 1989.

\bibitem{D-S} \textsc{F. Diamond and J. Shurman}, A first course in modular forms,
Grad. Texts in Math. 228, Springer, New York, 2005.

\bibitem{Gee} \textsc{A. Gee}, Class invariants by Shimura's reciprocity law,
J. Th\'{e}or. Nombres Bordeaux 11 (1999), 45--72.

\bibitem{G-Z} \textsc{B. Gross and D. Zagier}, On singular moduli,
J. Reine Angew. Math. 355 (1985), 191--220.

\bibitem {Hasse} \textsc{H. Hasse}, Neue bergr\"{u}ndung der
komplexen multiplikation, Teil I, J. f\"{u}r Math. 157 (1927),
115--139, Teil II, ibid. 165 (1931), 64--88.

\bibitem{J-K-S} \textsc{H. Y. Jung, J. K. Koo and D. H. Shin},
Generation of ray class field by elliptic units, Bull. Lond. Math.
Soc. 41 (2009), 935--942.

\bibitem{K-S} \textsc{J. K. Koo and D. H. Shin}, On some arithmetic properties of Siegel functions,
Math. Zeit. 264 (2010), 137--177.

\bibitem{K-L} \textsc{D. Kubert and S. Lang}, Modular units, Grundlehren der mathematischen Wissenschaften 244, Spinger-Verlag,
New York-Berlin, 1981.

\bibitem{Lang} \textsc{S. Lang}, Elliptic functions,
With an appendix by J. Tate, 2nd edition, Grad. Texts in Math. 112,
Spinger-Verlag, New York, 1987.

\bibitem{Ramachandra} \textsc{K. Ramachandra}, Some applications of Kronecker's limit formula,
Ann. of Math. (2) 80 (1964), 104--148.

\bibitem{Schertz} \textsc{R. Schertz}, Construction of ray class fields by elliptic units,
J. Th\'{e}or. Nombres Bordeaux 9 (1997), 383--394.

\bibitem{Shimura} \textsc{G. Shimura}, Introduction to the arithmetic theory of automorphic
functions, Iwanami Shoten and Princeton University Press, Princeton,
N. J., 1971.

\bibitem{Silverman} \textsc{J. H. Silverman}, Advanced topics in the
arithmetic of elliptic curves, Grad. Texts in Math. 151,
Springer-Verlag, New York, 1994.

\bibitem{Stevenhagen} \textsc{P. Stevenhagen}, Hilbert's 12th problem,
complex multiplication and Shimura reciprocity, Class field
theory-its centenary and prospect (Tokyo, 1998), 161--176, Adv.
Stud. Pure Math. 30, Math. Soc. Japan, Tokyo, 2001.

\end{thebibliography}
\end{document}